\documentclass[11pt]{article}

\usepackage{amsfonts,amsmath}
\usepackage{amssymb}

\newtheorem{theorem}{Theorem}[section]
\newtheorem{lemma}{Lemma}[section]
\newtheorem{proposition}[theorem]{Proposition}

\newtheorem{definition}{Definition}[section]
\newtheorem{preexample}{Example}[section]
\newenvironment{example}{\begin{preexample}}{\end{preexample}}
\newtheorem{preremark}{Remark}
\newenvironment{remark}{\begin{preremark}\rm}{\end{preremark}}

\newcommand{\qed }{ \hfill $\Box$ }

\begin{document}

\begin{center}
{\Large Stochastic characterization of harmonic sections and a Liouville theorem\\}

\end{center}

\vspace{0.3cm}

\begin{center}
{\large  Sim\~ao Stelmastchuk \footnote{The research of  S. Stelmastchuk is partially supported by FAPESP 02/12154-8.}
}  \\

\textit{Departamento de  Matem\'{a}tica, Universidade Estadual de Campinas,\\ 13.081-970 -  Campinas - SP, Brazil.
e-mail: simnaos@gmail.com}
\end{center}

\vspace{0.3cm}

\begin{abstract}
Let $P(M,G)$ be a principal fiber bundle and $E(M,N,G,P)$ be an associate fiber bundle. Our interested is to study harmonic sections of the projection $\pi_{E}$ of $E$ into $M$. Our first purpose is to give a stochastic characterization of harmonic section from $M$ into $E$ and a geometric characterization of harmonic sections with respect to its equivariant lift. The second purpose is to show a version of Liouville theorem for harmonic sections and to prove that section $M$ into $E$ is a harmonic section if and only if it is parallel.
\end{abstract}


\vspace{0.3cm} \noindent {\bf MSC2010 subject classification:} 53C43, 55R10, 58E20, 58J65, 60H30

\section{Introduction}

Let $\pi_{E}:(E,k) \rightarrow (M,g)$ be a Riemannian submmersion and $\sigma$ be a section of $\pi_{E}$, that is,
$\pi_{E} \circ M = Id_{M}$. We know that $TE= VE \oplus HE$ such that $VE=\rm{ker}(\pi_{E*})$ and $HE$ is the
horizontal bundle ortogonal to $VE$. C. Wood has studied the harmonic sections in many context, see \cite{wood1},
\cite{wood2}, \cite{wood3}, \cite{wood4}, \cite{wood5}. To recall, a harmonic sections is a minimal section for the
vertical energy functional
\[
 E(\sigma) = \frac{1}{2}\int_{M} \| \mathbf{v}\sigma_{*} \|^{2}vol(g),
\]
where $\mathbf{v}\sigma_{*}$ is the vertical component of $\sigma_*$. Furthermore, in \cite{wood1}, Wood showed that
$\sigma$ is a minimizer of the vertical energy functional if
\[
 \tau^{v}_{\sigma} = \rm{tr}\nabla^{v}\mathbf{v}\sigma_{*} = 0,
\]
where $\nabla^{v}$ is the vertical part of Levi-Civita connection on $E$, since $\pi_{E}$ has totally geodesics fibers.
Wood called $\sigma$ a harmonic section if $\tau_{\sigma}^{v} = 0$.

In this work, we drop the Riemannian submersion condition of $\pi_{E}$ and we mantain the fact that $TE = VE \oplus HE$
and that $M$ is a Riemmanian Manifold. Let $\nabla^{E}$ be a symmetric connection on $E$, where $E$ is not necessarily
a Riemannian manifold. About these conditions we can define harmonic sections in the same way that Wood, only observing
that $\nabla^{v}$ is vertical connection induced by $\nabla^{E}$. There is no compatibility between $\nabla^{E}$ and
Levi-Civita connection on $M$. The only assumption about $\nabla^{E}$ is that $\pi_{E}$ has totally geodesic fibres.

Furthermore, we restrict the context of our study. Let $P(M,G)$ be a Riemannian $G$-principal fiber bundle over a
Riemannian manifold $M$ such that the projection $\pi$ of $P$ into $M$ is Riemmanian submmersion. Suppose that $P$ has
a connection form $\omega$. Let $E(M,N,G,P)$ be an associated fiber bundle of $P$ with fiber $N$. It is well know that
$\omega$ yields horizontal spaces on $E$. Our goal is to study the harmonic sections of projection $\pi_{E}$.

Let $F:P \rightarrow N$ be a differential map. We call $F$ a horizontally harmonic map if $\tau_{F} \circ (H \otimes H) = 0$, where $H$ is the horizontal lift from $M$ into $P$ associated to $\omega$.

Let $\sigma$ be a section of $\pi_{E}$. It is well know that there exists a unique equivariant lift $F_{\sigma}: P \rightarrow N$ associated to $\sigma$. Our first purpose is to give an stochastic characterization for the harmonic section $\sigma$ and the horizontally harmonic map $F_{\sigma}$. From these stochastic characterizations we show that a section $\sigma$ of $\pi_{E}$ is harmonic section if and only if $F_{\sigma}$ is a horizontally harmonic map. This result is an extension of Theorem 1 in \cite{wood1}. As an example of application of stochastic characterization of harmonic sections, we will characterize the harmonic sections on tangent bundle with complete and horizontal lifts.

For our second purpose we consider $P(M,G)$ endowed with the Kaluza-Klein metric, $M$ and $G$ with the  Brownian coupling property and $N$ with the non-confluence property. About these conditions we show a version of Liouville Theorem and a version of result due to T. Ishiara in \cite{ishihara} to harmonic sections. As applications of our Liouville Theorem we can show the following. If we suppose that $M$ is complete Riemmanian manifold with nonnegative Ricci curvature and its tangent bundle $TM$ is endowed with the Sasaky metric, then the harmonic sections $\sigma$ of $\pi_{TM}$ are the $0$-section. In the same way we can construct an ambient for Hopf fibrations, with Riemannian structure, such that harmonic sections are the 0-section.
%
%

\section{Preliminaries}

In this work we use freely the concepts and notations of  P. Protter \cite{Protter}, E. Hsu \cite{hsu}, P. Meyer
\cite{Mey3}, M. Emery \cite{emery1} and \cite{emery3}, W. Kendall \cite{Kendall2} and S. Kobayashi and N. Nomizu
\cite{kobay}. We refer the reader to \cite{catuogno3} for a complete survey about the objects of this section.

Let $(\Omega, \mathcal{F},(\mathcal{F}_{t})_{t\geq0}, \mathbb{P})$ be a probability space which satisfies the usual
hypothesis (see for example \cite{emery1}). Our basic assumptions is that every stochastic process are continuos.

\begin{definition}
Let $M$ be a differential manifold. Let $X$ be a process stochastic with valued in $M$. We call $X$ a semimartingale
if, for all $f$ smooth on $M$, $f(X)$ is a real semimartingale.
\end{definition}

Let $M$ be a differential manifold endowed whit symmetric connection $\nabla^{M}$. Let $X$ be a semimartingale in $M$ and $\theta$ be a 1-form on $M$ defined along $X$. Let $(x_{1},\ldots, x_{n})$ be a local coordinate system on $M$. We define the It\^o integral of $\theta$ along $X$, locally, by
\[
\int_0^t\theta d^{\nabla^{M}} X_{s} = \int_{0}^{t} \theta_{i}(X_{s})dX^{i}_{s} + \frac{1}{2}\int_{0}^{t} \Gamma_{jk}^{i}(X_{s})\theta_{i}(X_{s})d[X^{j},X^{k}]_{s},
\]
where $\theta = \theta_{i} dx^{i}$ with $\theta_{i}$ smooth functions and $\Gamma_{jk}^{i}$ are the Cristoffel symbol of connection $\nabla^{M}$. Let $b \in T^{(2,0)}M$ defined along $X$. We define the quadratic integral on $M$ along $X$, locally, by
\[
\int_0^{t}b\;(dX,dX)_{s} = \int_{0}^{t} b_{ij}(X_{s})d[X^{i},X^{j}]_{s},
\]
where $b = b_{ij} dx^{i} \otimes dx^{j}$ with  $b_{ij}$ smooth functions. Furthermore, we denote the Stratonovich integral by $\int_{0}^{t} \theta \delta X_{s}$.

Let $M$ and $N$ be differential manifolds endowed with symmetric connections $\nabla^{M}$ and $\nabla^{N}$,
respectively.  Let $F:M\rightarrow N$ be a differential map and $\theta$ be a section of $TN^*$. We have the following
geometric It\^o formula:
\begin{equation}\label{ito-ito}
\int_0^t\theta \;d^{\nabla^{N}} F(X_{s})=\int_0^tF^*\theta \;d^{\nabla^{M}} X_{s} +\frac{1}{2}\int_0^t\beta_F^*\theta
\;(dX,dX)_{s},
\end{equation}
where $\beta_F$ is the second fundamental form of $F$ (see \cite{catuogno3} or \cite{vil}  for the definition of
$\beta_{F}$). It is well known that $F$ is affine map if $\beta_{F} \equiv 0$.

\begin{definition}
Let $M$ be a differential manifold endowed with symmetric connection $\nabla^{M}$. A semimartingale $X$ with values in
$M$ is called a $\nabla^{M}$-martingale if $\int_0^t \theta \;d^ {M} X_{s}$ is a real local martingale for all $\theta
\in \Gamma(TM^*)$.
\end{definition}

\begin{definition}
Let $M$ be a Riemannian manifold equipped with metric $g$. Let $B$ be a semimartingale with values in $M$, we say that
$B$ is a $g$-Brownian motion in $M$ if $B$ is a $\nabla^{g}$-martingale, where $\nabla^{g}$ is the Levi-Civita
connection of $g$, and for any section $b$ of $T^{(2,0)}M$ we have that
\begin{equation}\label{Brownian}
\int_0^tb(dB,dB)_{s}=\int_0^t \rm{tr}\;b_{B_s}ds.
\end{equation}
\end{definition}

From (\ref{ito-ito}) and (\ref{Brownian})  we deduce the useful formula:
\begin{equation}\label{harmonic}
\int_0^t\theta d^{\nabla^{N}} F(B_{s})=\int_0^tF^*\theta d^{\nabla^{g}} B_{s} +\frac{1}{2}\int_0^t\tau_F^*\theta_{B_s}
ds,
\end{equation}
where $\tau_F$ is the tension field of $F$.

From formula (\ref{Brownian}) and Doob-Meyer decomposition it follows that $F $ is an harmonic map if and only if it
sends $g$-Brownian motions to $\nabla^{N}$-martingales.

\begin{definition}
Let $M$ be a differential manifold endowed with symmetric connection $\nabla^{M}$. $M$ has the non-confluence of
martingales property if for every filtered space $(\Omega,\mathcal{F},(\mathcal{F}_{t})_{t \geq 0}, \mathbb{P})$,
$M$-valued martingales $X$ and  $Y$ defined over $\Omega$ and every finite stopping time  $T$ such that
\[
X_{T}=Y_{T} \ \ a.s. \textrm{ we have } X = Y \ \
\textrm{over} \ \ [0, T].
\]
\end{definition}
\begin{example}\label{Bex1}
Let $M=V$ be a $n$-dimensional vector space with flat connection $\nabla^{n}$. Let $X$ and $Y$ be $V$-valued
martingales. Suppose that there are a stopping time $\tau$ with respect to $(\mathcal{F}_{t})_{t \geq 0}$, $K>0$ such
that $\tau \leq K <\infty$ and $X_{\tau} = Y_{\tau}$. Then straightforward calculus shows that $X_{t} = Y_{t}$ for $t
\in [0, \tau]$.
\end{example}

\begin{definition}
A Riemmanian manifold $M$ has the Brownian coupling property if for all $x_{0}, y_{0} \in M$ we can construct a
complete probability space $(\Omega, \mathcal{F}, \mathbb{P})$, a filtration $({\mathcal{F}}_t; t \geq 0)$ and two
Brownian motions $X$ and $Y$, not necessarily independents, but both adapted to filtration such that
\[
X_{0} = x_{0}, Y_{0} = y_{0}
\]
and
\begin{equation*}\label{acoplagem}
\mathbb{P}(X_{t} = Y_{t} \textrm{ for some  } t \geq 0 ) = 1.
\end{equation*}
The stopping time $T(X,Y) = \inf \{ t> 0 ; X_{t} = Y_{t}\}$ is called coupling time.
\end{definition}

\begin{example}\label{Bex2}
Let $M$ be a complete Riemannian manifold. In \cite{Kendall1}, W. Kendall has showed that if $M$ is compact or $M$ has
nonnegative Ricci curvature then $M$ has the Brownian coupling property.
\end{example}

Let $M$ be a Riemmanian manifold with metric $g$. Consider $X$ and $Y$ two $g$-Brownian motion in $M$ which satisfies
the Brownian coupling property and $X_{0}=x, Y_{0}=y$, where $x,y \in M$. Denote by $T(X,Y)$ their coupling time. The
process $\bar{Y}$ is defined by
\begin{equation}\label{coalescenciageral}
\bar{Y}_{t} = \left\{
\begin{array}{rcl}
Y_{t} & , & t \leq T(X,Y)\\
X_{t} & , & t \geq T(X,Y).
\end{array}
\right.
\end{equation}
It is imediatelly that $\bar{Y}_{0}= y_{0}$.

\begin{proposition}\label{Bprop1}
Let $M$ be a Riemannian manifold with metric $g$. Suppose that $M$ has the Brownian coupling property. Let $X,Y$ be two
$g$-Brownian motions in $M$ which satisfies the Brownian coupling property. Then the process $\bar{Y}$ is a
$g$-Brownian motion in $M$.
\end{proposition}
\begin{proof}
It is a straightforward proof from definition of Brownian motion.
\qed
\end{proof}

%
%
\section{Harmonic sections}

Let $P(M,G)$ be a principal fiber bundle over $M$ and $E(M,N,G,P)$ be an associate fiber bundle to $P(M,G)$. We denote
the canonical projection from $P\times N$ into $E$ by $\mu$, namely, $\mu(p,\xi)=p\cdot\xi$. For each $p \in P$, we
have the map $\mu_{p}:N \rightarrow E$ defined by  $\mu_{p}(\xi) = \mu(p,\xi)$. Let $\sigma:E \rightarrow M$ be a
section of projection $\pi_{E}$, that is, $\pi_{E} \circ \sigma = Id_{M}$. There exists a unique equivariant lift
$F_{\sigma}:P \rightarrow N$ associated to $\sigma$ which is defined by
\begin{equation}\label{equivariantlift}
F_{\sigma}(p) = \mu_{p}^{-1} \circ \sigma \circ \pi (p).
\end{equation}
The equivariance property of $F_{\sigma}$ is given by
\[
F_{\sigma}(p \cdot g) = g^{-1} \cdot F_{\sigma}(p), \ \ \ g \in G.
\]

Let us endow $P$ and $M$ with Riemmanian metrics $k$ and $g$, respectively, such that $\pi:(P,k) \rightarrow (M,g)$ is
a Riemmanian submmersion. Let $\omega$ be a connection form on $P$. We observe that the connection form $\omega$ yields
a horizontal structure on $E$, that is, for each $b \in E$, $T_{b}E = V_{b}E \oplus H_{b}E$, where $V_{b}E:=
\mathrm{Ker}(\pi_{Eb*})$ and $H_{b}E$ is the horizontal subspace done by $\omega$ on $E$ (see for example \cite{kobay},
pp.87). We denote by $\mathbf{v}:TE \rightarrow VE$ and $\mathbf{h}: TE \rightarrow HE$ the vertical and horizontal
projection, respectively.

Let $\nabla^{M}$ denote the Levi-Civita connection on $M$ and $\nabla^{E}$ be a symmetric connection on $E$.
We are interested in connections $\nabla^{E}$ such that the projection $\pi_{E}$ from $E$ into $M$ has totally geodesic
fibres.

We denote by $\nabla^{v}$ the vertical component of connection $\nabla^{E}$ on $TE$, that is, for $X$ a vector field
and $V$ a vertical vector field on $E$ we have
\[
\nabla^{v}_{X}V = \mathbf{v}\nabla^{E}_{X}V.
\]
Let us denote $\nabla^{x}$ the induced connection of $\nabla^{v}$ over fiber $\pi_{E}^{-1}(x)$ for all $x \in M$. We
endow $N$ with a connection $\nabla^{N}$ such that, for each $p \in P$, $\mu_{p}$ is an affine map over its image, the
fiber $\pi_{E}^{-1}(x)$ with $\pi(p)=x$.

As $\pi_{E}$ has totally geodesics fibres we have, for each $p \in P$, $\beta^{v}_{\mu_{p}} = \beta^{x}_{\mu_{p}}$, where
\[
\beta^{v}_{\mu_{p}}= \nabla^{v}\circ \mu_{p*} - \mu_{p*}\nabla^{N}, \ \ \beta^{x}_{\mu_{p}}= \nabla^{x}\circ \mu_{p*} - \mu_{p*}\nabla^{N}
\]
and $\pi(p)=x$. Since $\mu_{p}$ is affine map, for each $p \in P$, we conclude that $\beta^{v}_{\mu_{p}} \equiv 0$.
In summary we have the following

\begin{lemma}\label{Sle1}
Let $\pi_{E}$ be the projection from $E$ into $M$. Let $\nabla^{E}$ be a symmetric connection on $E$ such that $\pi_{E}$ has totally geodesic fibres. Suppose that $N$ is endowed with a connection such that, for each $p \in P$, $\mu_{p}$ is an affine map. Then $\beta^{v}_{\mu_{p}} \equiv 0$ for all $p \in P$.
\end{lemma}

Let $\sigma$ be a section of $\pi_{E}$. Write $\sigma_{*} = \mathbf{v}\sigma_{*} + \mathbf{h}\sigma_{*}$, where $\mathbf{v}\sigma_{*}$ and $\mathbf{h}\sigma_{*}$ are the vertical and the horizontal component of $\sigma_{*}$, respectively. The second fundamental form for $\mathbf{v}\sigma_{*}$ is defined by
\[
 \beta_{\sigma}^{v} = \bar{\nabla}^{v} \circ \mathbf{v}\sigma_{*} - \mathbf{v}\sigma_{*} \circ \nabla^{M},
\]
where $\bar{\nabla}^{v}$ is the induced connection on $\sigma^{-1}E$. The vertical tension field is given by
\[
 \tau_{\sigma}^{v} = \rm{tr}\beta_{\sigma}^{v}.
\]

In the following we extend the definition given by C. M. Wood \cite{wood2} of harmonic section.

\begin{definition}
1. A section $\sigma$ of $\pi_{E}$ is called harmonic section if $\tau_{\sigma}^{v} =0$;\\
2. A differential map $F:P \rightarrow N$ is called horizontally harmonic if \linebreak $\tau_{F} \circ (H \otimes H) =
0 $, where $H$ is horizontal lift from $M$ into $P$.
\end{definition}

\begin{definition}
1. Let $\theta$ be a 1-form on $E$. We call $\theta$ a vertical form if $\theta \in VE^{*}$, the adjoint of vertical bundle $VE$. \\
2. A $E$-valued semimartingale $X$ is called a vertical martingale if, for every vertical form $\theta$ on $E$, $\int_{0}^{t} \theta d^{\nabla^{v}} X_{s}$ is a real local martingale.
\end{definition}

Now, we give a characterization of harmonic section in the context that we are working.

\begin{theorem}\label{Steo1}
Let $E(M,N,G,P)$ be an associated fiber bundle to principal fiber bundle $P(M,G)$, where $(M,g)$ is a Riemannian manifold. Let us endow $E$ with a symmetric connection on $\nabla^{E}$. Then a section $\sigma$ of $\pi_{E}$ is harmonic section if and only if, for every $g$-Brownian motion $B$ in $M$, $\sigma(B)$ is a vertical martingale.
\end{theorem}
\begin{proof}
Let $B$ be a $g$-Brownian motion in $M$ and $\theta$ be a vertical form on $E$. By formula (\ref{harmonic}),
\[
\int_{0}^{t} \theta~ d^{\nabla^{v}}\sigma(B_{s}) = \int_{0}^{t} \sigma^{*} \theta ~d^{\nabla^{M}}B_{s}
+ \frac{1}{2}\int_{0}^{t} \tau_{\sigma}^{v*}\theta(B_{s})~ds.\\
\]
We observe that $\int \sigma^{*} \theta~d^{\nabla^{M}}B_{s}$ is a real local martingale. Since $B$ and $\theta$ are arbitraries, Doob-Meyer decomposition assure that $\int_{0}^{t} \theta~ d^{\nabla^{v}}\sigma(B_{s})$ is real local martingale if and only if $\tau_{\sigma}^{v}$ vanishes. From  definitions of vertical martingale and harmonic section we conclude the proof. \qed
\end{proof}
\begin{remark}
In this Theorem we observe that we do not need the hypotesis of principal fiber bundle be a Riemannian manifold and $\pi: P \rightarrow M$ be a Riemannian submmersion. It is possible to give this characterization in general way for a submersion.
\end{remark}

Before we follow we need to give the definition of horizontal semimartingale in the fiber bundle $P$. Furthermore, to
prove our next result it is necessary to writer a horizontal semimartingale in $P$ as coordinates of a local coordinate
system of $P$ .
\begin{definition}
Let $Y$ be a semimartingale in $P$. We say that $Y$ is a horizontal martingale if  $\int_{0}^{t} \omega \delta Y_{s} = 0$, where $\omega$ is the connection 1-form on $P$.
\end{definition}

It is clear that if $Y$ is a horizontal semimartingale in $P$ then $\delta Y_{t}$ in $HP$, the horizontal bundle given by the connection 1-form $\omega$. Let $(U\times V, x^{i},v^{l})$ be a local coordinate system in $P$. We wish to describe $Y$ in this coordinates. For this end, firstly, we need describe the generators for $H_{p}P$, with $p \in U \times V$. In fact, we observe that $(U, x^{i})$ is a local coordinate system in $M$, so we denote the coordinate vector fields $\partial / \partial x^{i}$ by  $D_{i}$, $i=1, \ldots, n$. Let $H_{p}: T_{\pi(p)}M \rightarrow T_{p}P$, $p \in P$, the family of linear isomorphism yielded by $\omega$. We recall that this family is called horizontal lift from $M$ to $P$ associated to $\omega$. Because $D_{i}$, $i=1, \ldots n$, is a local frame on $U$, we have that $HD_{i}$ is a horizontal local frame on $U \times V$. Let us denote $\bar{D}_{l}$ the coordinate vector field in $VP$, the vertical fiber bundle. To
obtain or goal is sufficient describe the vector $HD_{i}$ in terms of $D_{i}$ and $\bar{D}_{l}$. In  fact,
for $i = 1, \ldots, n$ a simple account shows that
\[
HD_{i} = D_{i} - \omega_{i}^{l}\bar{D}_{l}, \ \ i=1, \ldots n \textrm{ and}  \ \ l=1, \ldots m
\]
where $\omega^{i}_{l}$ are coordinates of $\omega(D_{i})$ in the Lie algebra $\mathfrak{g}$  and $m$ is the dimension of Lie group $G$ that acts in $P$. It follows that
\[
\delta Y_{t} = \delta Y_{t}^{i}HD_{i} = \delta Y_{t}^{i}D_{i} - \delta Y_{t}^{i} \omega_{i}^{l}\bar{D}_{l}.
\]
From this we conclude that $Y$ has the following coordinates in $U \times V$
\begin{equation}\label{horcoord}
Y_{t} = ( Y_{t}^{1}, \ldots, Y_{t}^{n}, -\int_{0}^{t} \delta Y_{s}^{i} \omega_{i}^{1}, \ldots , -\int_{0}^{t} \delta Y_{s}^{i} \omega_{i}^{m})
\end{equation}
We use this fact in the proof of our next Lemma, which is a key in the demonstration of our after Proposition.

\begin{lemma}\label{Sle2}
Let $X_{t}$ be a E-valued semimartingale such that $X_{t}=\mu(Y_{t},\xi_{t})$, where $Y_{t}$ is a horizontal
$P$-semimartingale and $\xi_{t}$ is a $N$-semimartingale. Then
\[
 \int_{0}^{t} \theta d^{\nabla^{v}}X_{s} = \int_{0}^{t} \theta_{\alpha} d^{\nabla^{v}} \mu_{Y_{s}}(\xi_{s})
\]
 for all vertical form $\theta$ on $E$, where $\mu_{Y_{t}}(\xi_{t})$ is the vertical semimartingale in $E$ associated to $X$.
\end{lemma}

\begin{proof}
Let $(U \times V \times W, x^{i}, v^{l}, \nu^{\alpha})$ be a local coordinate system in $P\times N$. Thus $(U,x^{i})$,
$(V, v^{l})$ and $(W, \nu^{\alpha})$ are local coordinate systems in $M$, $G$ and $N$, respectively. Let $X_{t}$ be a
$E$-semimartingale as assumption and $\theta$ a vertical form.
By definition of It\^o integral,
\begin{equation}\label{Sle1eq1}
 \int_{0}^{t} \theta d^{\nabla^{v}}X_{s} = \int_{0}^{t} \theta_{\alpha}(X_{s}) d X^{\alpha}_{s} +\frac{1}{2}\int_{0}^{t} \Gamma^{\alpha}_{\beta \gamma}(X_{s}) \theta_{\alpha}(X_{s}) d[X^{\beta},X^{\gamma}]_{s},
\end{equation}
where $X^{\alpha}_{t}= \nu^{\alpha}(X_{t})$. It is sufficient to prove that $dX^{\alpha}_{t} =
d(\mu_{Y_{t}}\xi_{t})^{\alpha}$ for all $\alpha=1, \ldots, r$. Here, we observe that $\nu^{\alpha}$, $\alpha=1,\ldots,
n$, are vertical local coordinates in $E$ and it is not dependent of local coordinate system, because It\^o integral do
not depend. Let $\omega$ be a connection form on $P$ and $(U \times V, x^{i}, v^{l})$ be a local coordinate system in
$P$. Since $Y$ is a horizontal semimartingale in $P$, from (\ref{horcoord}) we see that
\[
Y = (Y^{1}, \ldots, Y^{n}, -\int dY^{i}\omega_{i}^{1}, \ldots, -\int dY^{i}\omega_{i}^{m}).
\]
Therefore, the process $(Y_{s},\xi_{s})$ can be written, in coordinates, as
\[
(Y^{1}, \ldots, Y^{n}, -\int dY^{i}\omega_{i}^{1}, \ldots, -\int dY^{i}\omega_{i}^{n},\xi^{1},\ldots,\xi^{r}).
\]
Let us denote $D_{i} = \partial/\partial x^{i}$, $\bar{D}_{l} = \partial/\partial v^{l}$ and $\tilde{D}_{\beta} = \partial/\partial \nu^{\beta}$,
for $i=1, \ldots, n$, $l=1, \ldots,m$ and $\beta=1, \ldots,r$. In the sequence, we consider the functions applied at $(Y_{t},\xi_{t})$. Applying
the It\^o formula in $v^{\alpha}\mu(Y_{t},\xi_{t})$ we obtain 
{\setlength\arraycolsep{2pt}
\begin{eqnarray}\label{Sle1eq2}
\lefteqn{dv^{\alpha}\mu(Y_{t},\xi_{t})
 =  D_{i}(v^{\alpha}\mu)dY^{i}_{t} + \bar{D}_{l}(v^{\alpha}\mu)d(\int dY^{i}\omega^{i}_{l})_{t} + \tilde{D}_{\beta}(v^{\alpha}\mu)d\xi^{\beta}_{t}}\nonumber \\
& + &\frac{1}{2}D_{j}D_{i}(v^{\alpha}\mu)d[Y^{i},Y^{j}]_{t}+ \frac{1}{2}\bar{D}_{l}D_{i}(v^{\alpha}\mu)d[-\int dY^{j}\omega^{j}_{l},Y^{j}]_{t}\nonumber\\
& + &\frac{1}{2}D_{i}\bar{D}_{l}(v^{\alpha}\mu)d[Y^{j},-\!\int dY^{j}\omega^{j}_{l}]_{t} + \frac{1}{2}\bar{D}_{k}\bar{D}_{l}(v^{\alpha}\mu)d[-\!\int dY^{j}\omega^{j}_{k},-\!\int dY^{j}\omega^{j}_{l}]_{t}\nonumber \\
& + &\frac{1}{2} D_{i}\tilde{D}_{\beta}(v^{\alpha}\mu)d[Y^{i},\xi^{\beta}]_{t} + \frac{1}{2}\tilde{D}_{\beta}D_{i}(v^{\alpha}\mu)d[\xi^{\beta},Y^{i}]_{t}\nonumber\\
& + &\frac{1}{2}\tilde{D}_{\beta}\bar{D}_{l}(v^{\alpha}\mu)d[\xi^{\beta},-\int dY^{j}\omega^{j}_{l}]_{t} + \frac{1}{2}\bar{D}_{l}\tilde{D}_{\beta}(v^{\alpha}\mu)d[-\int dY^{j}\omega^{j}_{l},\xi^{\beta}]_{t} \nonumber\\
& + &\frac{1}{2}\tilde{D}_{\beta}\tilde{D}_{\gamma}(v^{\alpha}\mu)d[\xi^{\beta},\xi^{\gamma}]_{t}.
\end{eqnarray}}
Now, we observe that $D_{i} -\omega_{i}^{l}\bar{D}_{l}$ is a horizontal vector field on $P$. Hence, for $\xi \in N$,
$\mu_{\xi*}(D_{i} -\omega_{i}^{l}\bar{D}_{l})$ is a horizontal vector field on $E$. However $dv^{\alpha}$ is a vertical
form on $E$ for all $\alpha=1, \ldots,r$. It follows that $dv^{\alpha}\mu_{\xi*}(D_{i}
-\omega_{i}^{l}\bar{D}_{l})\equiv 0$. Applying these observations in the equality above we conclude that
\[
dv^{\alpha}\mu(Y_{t},\xi_{t}) = \tilde{D}_{\beta}(v^{\alpha}\mu)d\xi^{\beta}_{t} + \frac{1}{2}\tilde{D}_{\beta}\tilde{D}_{\gamma}(v^{\alpha}\mu)d[\xi^{\beta},\xi^{\gamma}]_{t}.
\]
It follows that
\[
 dX^{\alpha}_{t} = dv^{\alpha}\mu(Y_{t},\xi_{t}) = dv^{\alpha}\mu_{Y_{t}}\xi_{t} = d(\mu_{Y_{t}}\xi_{t})^{\alpha}.
\]
Applying this equality in (\ref{Sle1eq1}) we conclude that
\[
\int_{0}^{t} \theta d^{\nabla^{v}}X_{s} = \int_{0}^{t}\theta d^{\nabla^{v}}\mu_{Y_{s}}\xi_{s}.
\]
\qed
\end{proof}

\begin{remark}
In the demonstration above one could think that the local coordinate system is relevant, but if we see (\ref{Sle1eq2})
as the second vector field of $P\times N$-semimartingale $(Y_{t},\xi_{t})$ applied to function $\nu^{\alpha}$ then from
Schwartz principle the second vector field $(Y_{t},\xi_{t})$ do not depend of local coordinate system.
\end{remark}
\begin{remark}
The Lemma above shows the strength of It\^o integral and  It\^o formula. One can try to use the Stratonovich-It\^o
conversion formula ( see for example \cite{cat-stel}) to conclude the same Lemma above. For this end, one will need
some restriction about connection $\nabla^{E}$. In fact, we need that the vertical part of $\nabla_{V}^{E}X$ and
$\nabla_{X}^{E}V$ vanishes, where $V$ and $X$ are vertical and horizontal vector fields, respectively, on $E$.
\end{remark}

Now, we relate the geometric and stochastic concepts of harmonic section and horizontally harmonic map.

\begin{proposition}\label{Sprop1}
Let $P(M,G)$ be a Riemannian principal fiber bundle endowed with a connection form $\omega$ and $M$ a Riemannian manifold such that the projection $\pi$ of $P$ into $M$ is a Riemannian submmersion. Let $E(M,N,G,P)$ be an associated fiber to $P$ endowed with a  symmetric connection $\nabla^{E}$ such that the projection $\pi_{E}$ has totally geodesic fibres. Moreover, suppose that $N$ has a connection $\nabla^{N}$ such that $\mu_{p}$ is an affine map for each $p \in P$.  Then
\begin{description}
\item[(i)] a $E$-valued semimartingale $X$ is vertical martingale if and only if $\mu_{Y}^{-1}\circ X$ is a
$\nabla^{N}$- martingale in $N$, where $Y=\pi_{E}(X)^{h}$ is the horizontal lift of $\pi_{E}(X)$ to  $P$;
\item[(ii)] a equivariant lift $F_{\sigma}$ associated to $\sigma$, $\sigma$ a section of $\pi_{E}$,  is horizontally harmonic map if
and only if, for every horizontal Brownian motion $B^{h}$ in $P$, $F_{\sigma}(B^{h})$ is a $\nabla^{N}$-martingale.
\end{description}
\end{proposition}
\begin{proof}
\noindent{\bf (i)} Let $X$ be a semimartingale in $E$ and $\theta$ be a vertical form on $E$. Let us denote $\xi =
\mu_{Y}^{-1}\circ X$. As $X = \mu(Y,\xi)$ we have, by Lemma \ref{Sle2},
\[
\int_{0}^{t} \theta d^{\nabla^{v}} X_{s} = \int_{0}^{t} \theta d^{\nabla^{v}} \mu_{Y_{s}}\xi_{s}.
\]
By geometric It\^o formula (\ref{ito-ito}),
\[
\int_{0}^{t} \theta d^{\nabla^{v}} X_{s}  = \int_{0}^{t} \mu_{Y_{s}}^{*}\theta d^{\nabla^{N}} \xi_{s} +
\frac{1}{2}\int\beta^{v*}_{\mu_{Y_{s}}}\theta(d\xi_{s},d\xi_{s}).
\]
From Lemma \ref{Sle1} we deduce that
\[
\int_{0}^{t} \theta d^{\nabla^{v}} X_{s}  = \int_{0}^{t} \mu_{Y_{s}}^{*}\theta d^{\nabla^{N}} \xi_{s}.
\]
So we conclude that $\int_{0}^{t} \theta d^{\nabla^{v}} X_{s}$ is local martingale if and only if $\int_{0}^{t}
\mu_{Y_{s}}^{*}\theta d^{\nabla{N}}\xi_{s}$ is too, and proof is complete.

\noindent{\bf (ii)} Let $B$ be a $g$-Brownian motion in $M$ and $B^{h}$ be a horizontal Brownian motion in $P$, that is,
\begin{equation}\label{igualdadehorizontal}
dB^{h} = H_{B}dB,
\end{equation}
where $H$ is the horizontal lift of $M$ to $P$. Set $\theta \in \Gamma(TN^{*})$. By geometric It\^o formula
(\ref{ito-ito}),
\[
\int_{0}^{t} \theta~ d^{\nabla^{N}}F_{\sigma}(B^{h}_{s}) = \int_{0}^{t} F_{\sigma}^{*} \theta ~d^{\nabla^{P}}B^{h}_{s}+
\int_{0}^{t} \beta_{F_{\sigma}}^{*} \theta(dB^{h},dB^{h})_{s}.
\]
From (\ref{igualdadehorizontal}) we see that
\[
\int_{0}^{t} \theta~ d^{\nabla^{N}}F_{\sigma}(B^{h}_{s}) = \int_{0}^{t} H^{*} F_{\sigma}^{*} \theta
~d^{\nabla^{M}}B_{s} + \int_{0}^{t} \beta_{F_{\sigma}}^{*} \theta (H_{B}dB,H_{B}dB)_{s}.
\]
As $B$ is Brownian motion we have
\[
\int_{0}^{t} \theta~ d^{\nabla^{N}}F_{\sigma}(B^{h}_{s}) = \int_{0}^{t} H^{*} F_{\sigma}^{*} \theta
~d^{\nabla^{M}}B_{s} + \int_{0}^{t} (\tau_{F_{\sigma}}^{H})^{*}\theta(B_{s}) ds,
\]
where $\tau_{F_{\sigma}}^{H} = \tau_{F_{\sigma}}\circ (H \otimes H)$. Since $\theta$ and $B$ are arbitraries, Doob-Meyer decomposition shows that $\int_{0}^{t} \theta d^{\nabla^{N}}F_{\sigma}(B^{h}_{s})$ is real local martingale if and only if $\tau_{F_{\sigma}}^{H}$ vanishes. From definitions of martingale and horizontally harmonic map we conclude the proof. \qed
\end{proof}

\begin{remark}
In the equation (\ref{igualdadehorizontal}) we can see the necessary of hypotesys of Riemannian submmersion over $\pi:P \rightarrow M$. In fact, it is well-know that the horizontal Brownian motion is defined as Stratonovich stochastic equation, see for example \cite{Shig}. However the Corollary 16 in \cite{emery2} shows that Stratonovich and It\^o differential equations are equivalent because the horizontal lift of geodesic in $M$ is a geodesic in $P$, since $\pi$ is a Riemannian submmersion.
\end{remark}

Now we give an extension of the characterization of harmonic sections obtained by C.M. Wood, see Theorem 1 in \cite{wood2}.

\begin{theorem}\label{Steo2}
Under the hypotheses of Proposition \ref{Sprop1}, a section $\sigma$ of $\pi_{E}$ is harmonic section if and only if $F_{\sigma}$ is horizontally harmonic map.
\end{theorem}
\begin{proof}
Let $B$ be a arbitrary $g$-Brownian motion in $M$ and $B^{h}$ be a horinzontal lift of $B$ in $P$, see equation
(\ref{igualdadehorizontal}).

Suppose that $\sigma$ is a harmonic section. Theorem \ref{Steo1}, shows that $\sigma(B)$ is a vertical martingale. But $\mu_{B^{h}}^{-1} \circ \sigma(B)$ is a $\nabla^{N}$-martingale, which follows from Proposition \ref{Sprop1}, item (i). Since $F_{\sigma}(B^{h}) = \mu_{B^{h}}^{-1} \circ \sigma \circ \pi (B^{h})$, it follows that $F_{\sigma}(B^{h})$ is a $\nabla^{N}$-martingale. Finally, Proposition \ref{Sprop1}, item (ii), shows that $F_{\sigma}$ is horizontally harmonic map.

Conversely, suppose that $F_{\sigma}$ is a horizontally harmonic map. Proposition \ref{Sprop1}, item (ii), shows that $F_{\sigma}(B^{h})$ is a $\nabla^{N}$-martingale. Since  \linebreak $F_{\sigma}(B^{h}) = \mu_{B^{h}}^{-1} \circ \sigma \circ \pi (B^{h})$, it follows that $\mu_{B^{h}}^{-1} \circ \sigma(B)$ is a $\nabla^{N}$-martingale. From Proposition \ref{Sprop1}, item (i), we see that  $\sigma(B)$ is a vertical martingale. We conclude from Theorem \ref{Steo1} that $\sigma$ is a harmonic section.
\qed\\
\end{proof}

%
%

\section{A Liouville theorem for harmonic sections }

We begin this section defining the Kaluza-Klein metric on $P(M,G)$. Let $P(M,G)$ be a principal fiber bundle endowed with a  connection form $\omega$, $M$ be a Riemannian manifold with a metric $g$ and $h$ be a bi-invariant metric on $G$. The Kaluza-Klein metric is defined by
\begin{equation} \label{kaluzakleinmetric}
k = \pi^{*}g + \omega^{*}h.
\end{equation}
From now on $P(M,G)$ is endowed with the Kaluza-Klein metric.

We will denote by $d_{P}$ and $d_{G}$ the Riemannian distance of $P$ and $G$, respectively.

\begin{lemma}\label{Lprop2}
Let $P(M,G)$ be a principal fiber bundle whit a Kaluza-Klein metric $k$, where $g$ is the Riemannian metric on $M$ and $h$ is  the bi-invariant metric on $G$ associated to $k$. The following assertions are holds:
\begin{description}
\item [(i)] Let $\tau:[0,1] \rightarrow P$ be a differential curve such that $\tau(t) = u \cdot \mu(t)$ with $\tau(0)= u$ and $\mu(t) \in G$, then
\[
 \int_{0}^{1} k(\dot{\tau}(t), \dot{\tau}(t))^{\frac{1}{2}}dt = \int_{0}^{1} h(\dot{\mu}(t), \dot{\mu}(t))^{\frac{1}{2}}dt.
\]
\item[(ii)] Let $\tau:[0,1] \rightarrow P$ be a differential curve. If $\gamma$ is a curve in $M$ and if $\mu$ is a curve in $G$ such that $\tau =\gamma(t)^{h} \cdot \mu(t)$, then
\[
 \int_{0}^{1} k(\dot{\tau}(t), \dot{\tau}(t))^{\frac{1}{2}}dt \leq \int_{0}^{1} g(\dot{\gamma}(t), \dot{\gamma}(t))^{\frac{1}{2}}dt + \int_{0}^{1} h(\dot{\mu}(t), \dot{\mu}(t))^{\frac{1}{2}}dt
\]
\item[(iii)] Let $x \in M$  and $u, v, w \in \pi^{-1}(x)$. If $a, b$ are points in $G$ such that  $v= u\cdot a$ and $w = u \cdot b$, then
\[
d_{P}(v,w) = d_{G}(a, b).
\]
\end{description}
\end{lemma}
\begin{proof}
\noindent {\bf(i) } and {\bf (ii)} The proofs are straightforward.\\
\noindent{\bf(iii) }
Let $\tau:[0,1] \rightarrow P$ be a differential curve such that $\tau(0) = v$ and $\tau(1) = w$. Consider a curve $\gamma$ in $M$ such that $\pi(\tau)=\gamma$. There exists a differential curve $\mu$ in $G$ such that $\mu(0)= a$, $\mu(1)=b$ and $\tau = \gamma^{h} \cdot \mu$. We observe that $\gamma(0)=x$ and $\gamma(1)=x$. This gives $\int_{0}^{1} g(\dot{\gamma}(t), \dot{\gamma}(t))^{\frac{1}{2}}dt = 0$. Thus from item (i) and item (ii) we conclude that
\[
 \int_{0}^{1} k(\dot{\tau}(t), \dot{\tau}(t))^{\frac{1}{2}}dt = \int_{0}^{1} h(\dot{\mu}(t), \dot{\mu}(t))^{\frac{1}{2}}dt.
\]
Therefore it is only necessary to consider vertical curves. It follows that $d_{P}(v,w) = d_{P}(u\cdot a, u \cdot b) = d_{G}(a, b)$, by definition of Riemmanian distance.
\qed
\end{proof}

\begin{theorem}\label{Lteo1}
Let $P(M,G)$ be a principal fiber bundle equipped with Kaluza-Klein metric and $E(M,N,G,P)$ be an associated fiber to $P$.  Let $\nabla^{E}$ and $\nabla^{N}$ be connetions on $E$ and $N$, respectively, such that the projection $\pi_{E}$ has totally geodesic fibres and $\mu_{p}$ is an affine map for each $p \in P$. Moreover, if $N$ has the non-confluence martingales property and if $M$ and $G$ have the Brownian coupling property, then
\begin{description}
 \item[(i)] a section $\sigma$ of $\pi_{E}$ is harmonic section if and only if $F_{\sigma}$ is constante map;
\item [(ii)] the left action of $G$ into $N$ has a fix point if there exists a harmonic section  $\sigma$ of $\pi_{E}$;
\item[(iii)] a section $\sigma$ of $\pi_{E}$ is harmonic section if and only if $\sigma$ is parallel.
\end{description}
\end{theorem}

\begin{proof}
{\bf (i)}
We first suppose that $F_{\sigma}$ is a constante map. Then it is immediately that $\tau_{\sigma}^{v}= 0$, so $\sigma$ is harmonic section.

Conversely, the proof will be divided into two parts. Firstly, we found a suitable stopping time $\tau$. After, we use $\tau$ to prove that $F_{\sigma}$ is constant over $P$.

Choose $x,y \in M$ arbitraries. By assumption about $M$, there exists two $g$-Brownian motion $X$ and $Y$ in $M$ such that $X_{0} = x$ and $Y_{0} =y$, which satisfy the Brownian coupling property. Consequently, the coupling time $T(X,Y)$ is finite. Proposition \ref{Bprop1} now assures that the process
\begin{equation}\label{coalescenciaemM}
\bar{Y_{t}} =
\left\{
\begin{array}{rcl}
Y_{t} & , & t \leq T(X,Y)\\
X_{t} & , & t \geq T(X,Y)
\end{array}
\right.
\end{equation}
is a $g$-Brownian motion in $M$.

Let $a, b \in G$ be arbitraries points. Since $G$ has the Brownian coupling property, we have two $h$-Brownian motion $\mu$ and $\nu$ in $G$ such that $\mu_{0} = a$, $\nu_{0} =b$. Moreover, there is a finite coupling time $T(\mu, \nu)$. But the process
\begin{equation}\label{coalescenciaemG}
\bar{\nu_{t}} = \left\{
\begin{array}{rcl}
\nu_{t} & , & t \leq T(\mu,\nu)\\
\mu_{t} & , & t \geq T(\mu,\nu)
\end{array}
\right.
\end{equation}
is a $h$-Brownian motion in $G$, which follows from Proposition \ref{Bprop1}.

Set $u, v \in P$ such that $\pi(u) = x$ and $\pi(v)= y$. Consider two horizontal Brownian motion $X^{h}$ and $\bar{Y}^{h}$ in $P$ such that $X^{h}_{0}= u$ and $\bar{Y}^{h}_{0}=v$. Define $\tau= T(X,Y) \vee T(\mu, \nu)$. We claim that
\begin{equation}\label{acoplagembrowniana}
X^{h}_{t}\cdot \mu_{t} = \bar{Y}^{h}_{t}\cdot \bar{\nu}_{t},
\textrm{ a.s. } \forall \ \ t \geq \tau.
\end{equation}
In fact, we need consider two cases. First, suppose that $T(X,Y) \leq T(\mu, \nu)$. For all $t \geq T(\mu, \nu)$ we have
\begin{equation*}
d_{P}(X^{h}_{t}\cdot \mu_{t},\bar{Y}^{h}_{t}\cdot \bar{\nu}_{t})=d_{P}(X^{h}_{t}\cdot \mu_{t},\bar{Y}^{h}_{t}\cdot
\mu_{t}) = d_{P}(R_{\mu_{t}}X^{h}_{t},R_{\mu_{t}}\bar{Y}^{h}_{t}).
\end{equation*}
Since $k$ is the Kaluza-Klein metric, it follows that
\begin{equation*}
d_{P}(X^{h}_{t}\cdot \mu_{t},\bar{Y}^{h}_{t}\cdot \bar{\nu}_{t})  =  d_{M}(X_{t},\bar{Y}_{t}).
\end{equation*}
From (\ref{coalescenciaemM}) we conclude that (\ref{acoplagembrowniana}) is satisfied for all $t \geq T(\mu, \nu)$.

In the other side, suppose that $T(X,Y) \geq T(\mu, \nu)$. For all $t \geq T(X,Y)$, Lemma \ref{Lprop2}, item (iii), assures that
\begin{equation*}
d_{P}(X^{h}_{t}\cdot \mu_{t},\bar{Y}^{h}_{t}\cdot \bar{\nu}_{t})=
d_{P}(X^{h}_{t}\cdot \mu_{t},X^{h}_{t}\cdot \bar{\nu}_{t}) =
d_{G}(\mu_{t},\bar{\nu}_{t}).
\end{equation*}
From (\ref{coalescenciaemG}) we conclude that (\ref{acoplagembrowniana}) is satisfied for all $t \geq T(X,Y)$.

Setting $t \geq \tau$ we obtain $F_{\sigma}(X^{h}_{t}\cdot \mu_{t})=F_{\sigma}(\bar{Y}^{h}_{t}\cdot \bar{\nu}_{t})$. Since $F_{\sigma}$ is equivariant by right action,  $\mu_{t}^{-1} \cdot F_{\sigma}(X^{h}_{t} )= \bar{\nu}_{t}^{-1} \cdot F_{\sigma}(\bar{Y}^{h}_{t})$. Because $\mu_{t} = \bar{\nu}_{t}$ for $t\geq \tau$, we conclude that $F_{\sigma}(X^{h}_{t}) = F_{\sigma}(\bar{Y}^{h}_{t})$.

Since $\sigma$ is a harmonic section,  from Theorem \ref{Steo2} we see that $F_{\sigma}$ is a horizontally harmonic map. Proposition \ref{Sprop1} now shows that $F_{\sigma}(X^{h}_{t})$ and $F_{\sigma}(\bar{Y}^{h}_{t})$ are $\nabla^{N}$-martingales in $N$. Since $N$ has non-confluence martingales property,
\[
F_{\sigma}(X^{h}_{0} ) = F_{\sigma}(\bar{Y}^{h}_{0}).
\]
It follows immediately that $F_{\sigma}(u) = F_{\sigma}(v)$. Consequently, $F_{\sigma}$ is a constant map.

\noindent{\bf{(ii)}} Let $\sigma$ be a harmonic section of $\pi_{E}$. From item (i) there exists $\xi \in N$ such that $F_{\sigma}(p) = \xi$ for all $p \in P$. We claim that $\xi$ is a fix point. In fact, set $a \in G$. From equivariant property of $F_{\sigma}$ we deduce that
\[
 a \cdot \xi = a \cdot F_{\sigma}(p) = F_{\sigma}(p \cdot a^{-1}) = \xi.
\]

\noindent{\bf(iii)} Let $\sigma$ be a section of $\pi_{E}$. Suppose that $\sigma$ is parallel. Then $\sigma_{*}(X)$ is horizontal for all $X \in TM$ (see for example \cite{kobay}, pp.114). This gives  $\mathbf{v}\sigma_{*}(X)=0$. Then it is clear, by definition, that $\sigma$ is harmonic section.

Suppose that $\sigma$ is a harmonic section. From item (i) it follows that there exists $\xi \in N$ such that $F_{\sigma}(p) = \xi$ for all $p \in P$. By definition of equivariant lift,
\[
 \sigma(x) = \sigma \circ \pi(p) = \mu(p, \xi) = \mu_{\xi}(p), \ \ \pi(p)= x,
\]
where $\mu_{\xi}$ is an application from $P$ into $E$. Let $v \in T_{x}M$ and let $\gamma(t)$ be a curve in $M$ such that $\gamma(0) = x $ and $\dot{\gamma}(0) = v$. Then
\[
 \sigma_{*}(v) = \left.\dfrac{d}{dt}\right|_{0} \sigma \circ \gamma(t)= \left.\dfrac{d}{dt}\right|_{0} \mu_{\xi} \circ \gamma^{h}(t)= \mu_{\xi*}(\dot{\gamma}^{h}(0)),
\]
where $\gamma^{h}$ is the horizontal lift of $\gamma$ into $P$. Since $\dot{\gamma}^{h}(0)$ is horizontal vector in $P$, so is $\mu_{\xi*}(\dot{\gamma}^{h}(0))$ in $E$ (see for example \cite{kobay}, pp.87). Therefore $\sigma_{*}(v)$ is horizontal vector. So we conclude that $\sigma$ is parallel.
\qed\\
\end{proof}


%
%

\section{Examples}

In this section, we will give two applications. First, we will use the stochastic characterization of harmonic section, see Theorem \ref{Steo1}, to give the harmonic sections on tangent bundle with complete and horizontal lifts. After, from Theorem \ref{Lteo1} we will show that about geometric conditions the unique harmonic section on Tangent bundle with Sasaki metric is null. Finally, we will work with Hopf fibrations and harmonic sections.\\

\noindent{{\large{\it Tangent Bundle with Complete and Horizontal lifts}}\\

Let $M$ be a Riemmanian manifold and $TM$ its tangent bundle. Let $\nabla^{M}$ be a symmetric connection on $M$. It is clear that $TM$ is an associated fiber bundle to orthonormal frame bundles $OM$, which has, naturally, a Kaluza-Klein metric. Also, it is possible to prolong $\nabla^{M}$ to a connection on $TM$. A two well known ways are the complete lift $\nabla^{c}$ and horizontal lift $\nabla^{h}$ (see \cite{yano} for the definitions of $\nabla^{c}$ and $\nabla^{h}$). Let $X,Y$ be vector fields on $M$, so $\nabla^{h}$ satisfies the following equations:
\begin{equation}
\begin{array}{ccl}
\displaystyle \nabla^{c}_{X^{V}}Y^{V} & = & 0  \\
\nabla^{c}_{X^{V}}Y^{H} & = & 0  \\
\nabla^{c}_{X^{H}}Y^{V} & = & (\nabla_{X}Y)^{V} \\
\nabla^{c}_{X^{H}}Y^{H} & = & (\nabla_{X}Y)^{H} + \gamma(R(-,X)Y,
\end{array}
\begin{array}{ccl}
\displaystyle \nabla^{h}_{X^{V}}Y^{V} & = & 0  \\
\nabla^{h}_{X^{V}}Y^{H} & = & 0  \\
\nabla^{h}_{X^{H}}Y^{V} & = & (\nabla_{X}Y)^{V} \\
\nabla^{h}_{X^{H}}Y^{H} & = & (\nabla_{X}Y)^{H},
\end{array}
\end{equation}
where $R(-,X)Y$ denotes a tensor field $W$ of type (1,1) in $M$ such that $W(Z)=R(Z,X)Y$ for any $Z \in T^{(0,1)}(M)$, and $\gamma$ is a lift of tensor, which is defined at page 12 in \cite{yano}.

First, we observe that the vertical connection $\nabla^{v}$ is the same one for complete and horizontal lifts. In particular, for $U,V$ vertical vector fields on $TM$ we see that $\nabla^{v}_{U}V \equiv 0$.

We wish to prove the following characterization of harmonic sections for complete and horizontal lifts.

\begin{proposition}
Let $M$ be a Riemannian manifold, $\nabla^{M}$ the Levi-Civita connection and $TM$ its tangent bundle.
Let $\sigma$ be a section of $\pi_{TM}$. Then
\begin{enumerate}
\item If $TM$ is endowed with $\nabla^{c}$, then $\sigma$ is a harmonic section with respect to $\nabla^{c}$ if and only if $\mathbf{v}\sigma_{*} \equiv 0$.
\item If $TM$ is endowed with $\nabla^{h}$, then $\sigma$ is a harmonic section with respect to $\nabla^{h}$ if and only if $\mathbf{v}\sigma_{*} \equiv 0$.
\end{enumerate}
\end{proposition}
\begin{proof}
Let $\sigma$ be a section of $\pi_{TM}$ such that $\sigma$ is a harmonic section with respect to $\nabla^{c}$ or $\nabla^{h}$. Let $(U\times V,x^{i},v^{\alpha})$ be a local coordinate system on $TM$. We claim that $\sigma^{\alpha} =v^{\alpha} \circ \sigma$ is a constant function. In fact, for every Brownian motion $B_{t}$ in $M$, Theorem \ref{Steo1} assures that $\sigma(B_{t})$ is a vertical martingale in $TM$. In particular, suppose that $B_{0}=x \in U$ and denote $\tau = inf\{t: B_{t}(\omega) \notin U, \forall \, \omega \in \Omega\}$, that is, the first exit time from $U$. Since the vertical connection $\nabla^{v}$ on $TM$ is the same to $\nabla^{c}$ and $\nabla^{h}$, from definition of It\^o stochastic we have 
\[
\int_{0}^{t} d v^{\alpha} d^{\nabla^{v}}\sigma(B_{s}) = \int_{0}^{t} d\sigma^{\alpha}(B_{s})  + \frac{1}{2}\int_{0}^{t} \Gamma^{\alpha}_{\beta \gamma}(\sigma_{s}) d[\sigma^{\beta}(B),\sigma^{\gamma}(B)]_{s}.
\]
where $0\leq t \leq \tau$. By definition, $\nabla^{v}_{U}V \equiv 0$ for $U,V$ vertical vector fields on $TM$. Thus $\Gamma^{\alpha}_{\beta \gamma} = 0$ for $\beta,\gamma = 1 \ldots n$, where $n$ is the dimension of $M$. Therefore
\[
\int_{0}^{t} d v^{\alpha} d^{\nabla^{v}}\sigma(B_{s}) = \int_{0}^{t} d\sigma^{\alpha}(B_{s}) = \sigma^{\alpha}(B_{t})
\]
Being $\sigma(B_{t})$ a vertical martingale, it follows that $\sigma^{\alpha}(B_{t})$ is a real local martingale. Applying the expectation at $\sigma^{\alpha}(B_{t})$ we obtain
\[
 \mathbb{E}(\sigma^{\alpha}(B_{t})) = \mathbb{E}(\sigma^{\alpha}(B_{0})) =\sigma^{\alpha}(B_{0}).
\]
From this we conclude, almost sure, that $\sigma^{\alpha}(B_{t})$ is constant. Since $B_{t}$ is an arbitrary Brownian motion, it follows that $\sigma^{\alpha}$ is a constant function.

To prove 1. and 2., it is only necessary to observe that $\mathbf{v}\sigma_{*}$, in any local coordinate system $(U\times V,x^{i},v^{\alpha})$, is written in function of $\sigma^{\alpha}$ and to use the conclusion above.
\qed \\
\end{proof}


{\large {\it Tangent bundle with Sasaki metric}}\\

Let $M$ be a complete Riemannian manifold which is compact or has nonnegative Ricci curvature. Let $OM$ be the ortonormal frame bundle endowed whit the Kaluza-Klein metric. Let $TM$ be the tangent bundle equipped with the Sasaki metric $g_{s}$. Thus $\pi_{E}$ is a Riemannian submersion with totally geodesic fibers and, for each $p \in P$, $\mu_{p}$ is a isometric map (see for example \cite{musso}). From these assumptions and Examples \ref{Bex1} and \ref{Bex2} it follows that the hypotheses of Theorem \ref{Lteo1} are satisfied.

\begin{proposition}
 Under conditions stated above, if $\sigma$ is a harmonic section of  $\pi_{TM}$, then $\sigma$ is the 0-section.
\end{proposition}
\begin{proof}
Let $\sigma$ be a harmonic section of $\pi_{TM}$. By Theorem \ref{Lteo1}, item (i), there exists $\xi \in N$ such that $F_{\sigma}(u)= \xi$ for all $u \in P$. Moreover, by item (ii) $\xi$ is a fix point of left action of $O(n,\mathbb{R})$ into $\mathbb{R}^{n}$. We observe that $0 \in \mathbb{R}^{n}$ is the unique fix point to this left action. Thus get $F_{\sigma}(u) = 0$. Therefore $\sigma$ is the 0-section.
\qed\\
\end{proof}

{\large {\it Hopf fibration}}\\

Let $S^{1} \rightarrow S^{2n-1}\rightarrow \mathbb{CP}^{n-1}$ be a Hopf fibration. It is well know that \linebreak $S^{2n-1}(\mathbb{CP}^{n-1}, S^{1})$ is a principal fiber bundle.  We recall that $U(1) \cong S^{1}$. Let $\phi$ be  the aplication of $U(1) \times \mathbb{C}^{m}$ into $\mathbb{C}^{m}$ given by
\begin{equation}\label{acaoaesquerdadeU1emCm}
(g, (z_{1}, \ldots, z_{m})) \rightarrow g \cdot (z_{1}, \ldots,
z_{m}) = (gz_{1}, \ldots, gz_{m}).
\end{equation}
Clearly, $\phi$ is a left action of $U(1)$ into $\mathbb{C}^{m}$. Thus, we can consider $\mathbb{C}^{m}$ as standard fiber of associate fiber $E(\mathbb{CP}^{n-1},\mathbb{C}^{m},S^{1},S^{2n-1})$, where $E = S^{2n-1}\times_{U(1)} \mathbb{C}^{m}$. We are considering the canonical scalar product $<,>$ on $\mathbb{C}^{n}$ and the induced Riemannian metric $g$ on $\mathbb{CP}^{n-1}$. Since $U(1)$ is invariant by $<,>$, there exists one and only one Riemannian metric $\hat{g}$ on $E$ such that $\pi_{E}$ is a Riemannian submersion from $(E,\hat{g})$ to $(M,g)$ with totally geodesic fibers isometrics to $(N,<,>)$ (see for example \cite{vil}). From these assumptions and examples \ref{Bex1} and \ref{Bex2}  wee see that hypotheses of Theorem \ref{Lteo1} are holds.
\begin{proposition}
Under conditions stated above, if $\sigma$ is a harmonic section of $\pi_{E}$, then $\sigma$ is the 0-section.
\end{proposition}
\begin{proof}
We first observe that $(0, \ldots, 0)$ is the unique fix point to the left action (\ref{acaoaesquerdadeU1emCm}). Since $\sigma$ is harmonic section, from Theorem \ref{Lteo1} we see that  $F_{\sigma}$ is constant map and $F_{\sigma}(p)=(0, \ldots, 0)$ for all $p \in S^{2n-1}$. Therefore $\sigma$ is the 0-section.
\qed
\end{proof}

\end{document}